\documentclass[english]{siamltex}
\usepackage[pagebackref,pdfusetitle]{hyperref}
\usepackage{url}

\usepackage[T1]{fontenc}
\usepackage[latin9]{inputenc}
\usepackage{babel}
\usepackage{graphicx}
\usepackage{epstopdf}
\usepackage{amsmath}
\usepackage{amsfonts}
\usepackage{amssymb}

\newtheorem{thm}{Proposition}
\newtheorem{cor}[thm]{Corollary}
\newtheorem{prop}[thm]{Proposition}

\begin{document}

\title{Symplectic integrators for index one constraints}
\author{Robert I McLachlan%
\thanks{Institute of Fundamental Sciences, Massey University,
	Private Bag 11 222, Palmerston North 4442, New Zealand (r.mclachlan@massey.ac.nz, m.c.wilkins@massey.ac.nz)}
	 \and Klas Modin\thanks{Department of Mathematical Sciences, Chalmers University of Technology, Gothenburg, Sweden (klas.modin@chalmers.se)}\and Olivier Verdier\thanks{Department of Mathematical Sciences,
	NTNU,
	7491 Trondheim,
	Norway (olivier. verdier@math.ntnu.no)} \and Matt Wilkins$^*$}

\maketitle

\begin{abstract}\noindent
 We show that symplectic Runge--Kutta methods provide effective
 symplectic integrators for Hamiltonian systems with index one 
 constraints. These include the Hamiltonian description of variational
 problems subject to position and velocity constraints
 nondegenerate in the velocities, such as those arising in 
 sub-Riemannian geometry and control theory.
\end{abstract}

\begin{keywords} Symplectic integrators, differential-algebraic equations, index one systems, variational nonholonomic equations, vakonomic equations, optimal control problems\end{keywords}

\begin{AMS} 65L80, 65P10, 70H45, 49J15\end{AMS}

\pagestyle{myheadings}
\thispagestyle{plain}
\markboth{R I MCLACHLAN, K MODIN, O VERDIER, AND M WILKINS}{SYMPLECTIC INTEGRATORS FOR INDEX ONE CONSTRAINTS}

\section{Introduction: constrained Hamiltonian systems}
\label{sec:srconstraints}
There are several different types of constrained dynamical systems. First, the constraints themselves can be {\em holonomic} (depending only on position) or {\em nonholonomic} (depending on position and velocity, but not the derivative of a holonomic constraint). Nonholonomic constraints are associated with two main types of dynamical system (we adopt the terminology of \cite{bloch}):
\begin{remunerate}
\item The {\em dynamic nonholonomic equations}, also known as the Lagrange--d'Alem\-bert equations, that describe many mechanical systems in rolling and sliding contact. They have been the subject of several studies in geometric numerical integration (see, e.g., \cite{monforte,fe-bl-ol,mc-pe,mo-ve} and references therein), but as the equations are not in general Hamiltonian or variational and their geometric properties are not fully understood, there is no consensus as to their best discrete version.
\item The {\em variational nonholonomic equations}, also known as the vakonomic equations, that are the subject of this paper. They arise in two main contexts in dynamics. 
The first is from the {\em Lagrange problem}, that of finding curves $q(t)$ in the configuration manifold  $M$ that make the action
$$ \int_{t_0}^{t_1} L(q,\dot q)\, dt$$
stationary amongst all curves satisfying the fixed endpoint conditions $q(t_0)=q_0$, $q(t_1)=q_1$ and satisfying the nonholonomic constraints $g(q,\dot q)=0$. If the action has the form
$$ \int_{t_0}^{t_1} \left< \dot q,\dot q\right>_q \, dt$$
for some metric $\left<,\right>_q$ on $M$
and the constraint is that $q(t)$ is tangent to a maximally nonintegrable distribution on $M$, the Lagrange problem becomes
the {\em sub-Riemannian geodesic problem} that defines
the sub-Riemannian geometry of $M$ \cite{ledonne,montgomery}\footnote{Also known as Carnot geometry in France and nonholonomic Riemannian geometry in Russia \cite{vershik}; the metric induced by sub-Riemannian geodesics is known as the Carnot-Carath\'eodory metric.}, an active branch of geometry. Some Lie groups (the Carnot groups, of which the Heisenberg group is an example) have a natural sub-Riemannian structure. 
\end{remunerate}

The second is in optimal control problems. Under (quite weak) conditions, the Lagrange problem is equivalent to the optimal control problem $\min_u\int_{t_0}^{t_1} g(q,u)\, dt$ subject to $q(t_0)=q_0$, $q(t_1)=q_1$, and $\dot q = f(q,u)$---the control $u$ can be eliminated (\cite{bloch}, Thm 7.3.3).
Many applications of the variational nonholonomic equations---to kinematic sub-Riemannian optimal control problems, to control on semi-simple Lie groups and symmetric spaces, to the motion of a particle in a magnetic field, and to optimal control on Riemannian manifolds and Lie groups---are discussed in detail in Chapter 7 of \cite{bloch}. Many familiar situations---from parking a car (an example we model numerically in Section \ref{sec:srexample}), riding a bike, rolling a ball, to controlling a satellite or a falling cat controlling itself---are described using the variational nonholonomic equations.

A Hamiltonian formulation of the variational nonholonomic equations is to consider
\begin{equation}
\label{eq:1}
 J \dot z = \nabla H(z),\quad z\in C\subset \mathbb{R}^m
 \end{equation}
where $z\in \mathbb{R}^m$, $\omega :=  \frac{1}{2}dz\wedge J dz$ is a closed 2-form\footnote{We use vector notation in wedge products, writing $dq\wedge dp$ for $\sum_{i=1}^m dq_i\wedge dp_i$ and $dz\wedge J dz$ for $\sum_{i,j=1}^m J_{ij} dz_i \wedge dz_j$, where the dimension $m$ is determined from the context.}, 
$H\colon\mathbb{R}^m\to R$ is a Hamiltonian, and $C$ is a constraint
submanifold such that $i^*\omega$ (where $i\colon C\to\mathbb{R}^m$ is the inclusion of $C$ in $\mathbb{R}^m$) is nondegenerate, i.e., such that $(C,i^*\omega)$  is a symplectic manifold. The dynamics on $C$ depends only on the restricted Hamiltonian $i^*H$ and restricted symplectic form $i^*\omega$. Indeed,  systems with
holonomic constraints also take this form, with $z=(q,p)$, $\omega = dq\wedge dp$, and
$C= \{(q,p): h_i(q)=0,\ Dh_i(q)H_p(q,p) = 0,\ 1\le i\le k\}$
consisting of primary and secondary constraints; a nondegeneracy assumption ensures that $C$ is symplectic. The widely used {\sc rattle} method \cite{simhamdyn,mc-mo-ve-wi} provides a 
(class of) symplectic integrators for this case when $J$ is constant: it integrates in coordinates $z$ with Lagrange multipliers to enforce the constraints. However, it is striking that there are no known symplectic integrators for general constrained Hamiltonian systems of the form
of Eq. (\ref{eq:1}).

 In this paper we describe a class of symplectic integrators for a class of Hamiltonian systems of the form (\ref{eq:1}) containing constraints that can depend on both position and velocity. The systems are those of index 1 and the integrators are given in Propositions \ref{thm:symrk} and \ref{thm:4} which are our main results. The class includes the Hamiltonian description of 
the variational nonholonomic equations, including the sub-Riemannian geodesic equations, and
  we give this application first, in Propositions \ref{thm:lagham} and \ref{thm:laghamgeneral}, as it motivates the consideration of index 1 Hamiltonian systems. The construction is generalized to include both holonomic and nonholonomic constraints in Section \ref{sec:holo}. Sample applications are given to calculating the sub-Riemannian geodesics of a wheeled vehicle (the `parallel parking' problem) in Section \ref{sec:srexample} and of the Heisenberg group in Section \ref{sec:hexample}.
  
In the following proposition, the linear independence assumption on the constraints is equivalent to 
constraining the velocities to lie in an $(n-k)$-dimensional distribution of the tangent
space of the positions.

\begin{prop}\label{thm:lagham}
   Let $M$ be a symmetric nonsingular $n \times n$ mass matrix,
   $V\colon \mathbb{R}^n \rightarrow \mathbb{R}$ a smooth
   potential, $g_i:\mathbb{R}^n\to\mathbb{R}^n$, $i=1,\dots,k$ be  smooth functions whose values are linearly independent for all arguments, and $q $ be  a smooth extremal with fixed endpoints for the functional
   \begin{equation}
   \label{eq:action}
   S(q )=\int_{t_0}^{t_1} L(t, q , \dot{q }) dt = \int_{t_0}^{t_1} \left( \frac{1}{2} \dot{q }^T M \dot{q } - V(q )\right) dt
   \end{equation}
   subject to the constraints $g _i(q ) \cdot \dot{q } = 0, i = 1, \ldots, k$. Then
   \begin{equation}
   \label{eq:ham}
      J\dot{z } = \nabla H(z )
   \end{equation}
   where
   \begin{align*}
      J &= \begin{pmatrix}0 & - \mathbb{I}_{n \times n} & 0\\
         \mathbb{I}_{n \times n} & 0 & 0\\0 & 0 & 0_{k \times k}\end{pmatrix} ,\quad
   z  = \begin{pmatrix} q  \\ p  \\ \lambda   \end{pmatrix}\in\mathbb{R}^{2n+k}, \\
      p   &=  M \dot{q } - \sum_{i=1}^k \lambda_i g _i(q ), \\
      H(z ) &= \frac{1}{2} \left( p  + \sum_{i=1}^k \lambda_i g _i(q ) \right)^T M^{-1}  \left( p  + \sum_{i=1}^k \lambda_i g _i(q ) \right)  + V(q ), 
   \end{align*}
and, furthermore, the Euler--Lagrange equations for (\ref{eq:action}) are equivalent to the generalized Hamiltonian system (\ref{eq:ham}).
Eq. (\ref{eq:ham}) forms a constrained Hamiltonian system of the type (\ref{eq:1}) with constraint submanifold $C$ a graph
over $(q,p)$, i.e., $C := \{(q,p,\lambda)\colon \lambda=\tilde\lambda(q,p)\}$
and restricted symplectic form $i^*\omega = dq\wedge dp$.

\end{prop}

 \begin{proof}
Introducing Lagrange multipliers $\lambda_1,\dots,\lambda_k$,
the Euler--Lagrange equations for (\ref{eq:action}) are 
\begin{gather}
   \frac{d}{d t} \left( \nabla_{\dot{q }} F \right) - \nabla_{q } F = 0, \label{eqn:eulerlag}\\ 
   g _i(q ) \cdot \dot{q } = 0, \quad i = 1, \ldots, k, \label{eqn:constraints}
\end{gather}
where
\[
F(q,\dot q,\lambda) = \frac{1}{2} \dot{q }^T M \dot{q } - V(q ) -  \sum_{i=1}^k \lambda_i g _i(q ) \cdot \dot{q }.
\]
Expanding out equation~(\ref{eqn:eulerlag}) gives the Euler--Lagrange equations
\begin{align}
   \frac{d}{d t} \left( M \dot{q } - \sum_{i=1}^k \lambda_i g _i(q ) \right) - \nabla_{q } F &= 0, && \nonumber \\
   \frac{d}{d t} \left( M \dot{q } - \sum_{i=1}^k \lambda_i g _i(q ) \right) + \left( \nabla V(q ) + \sum_{i=1}^k \lambda_i Dg _i(q )  \dot{q } \right) &= 0. &&  
\end{align}

Define the conjugate momentum $p\in\mathbb{R}^n $ using the standard Legendre
transform
\begin{equation}
   p   := \nabla_{\dot{q }} F = M \dot{q } - \sum_{i=1}^k \lambda_i g _i(q ) \label{eqn:hamp}
\end{equation}
so that
\begin{equation} 
   \dot{q } = M^{-1} \left( p   + \sum_{i=1}^k \lambda_i g _i(q ) \right). \label{eqn:hamqdot} 
\end{equation}
Using equations~(\ref{eqn:hamp}) and~(\ref{eqn:hamqdot}) in equation~(\ref{eqn:eulerlag}) gives 
\begin{equation} 
   \dot{p} = -\nabla V(q ) - \sum_{i=1}^k \lambda_i Dg _i(q ) M^{-1} \left( p  + \sum_{j=1}^k \lambda_j g _j(q ) \right). \label{eqn:hamsecond} 
\end{equation}

Defining $H(q,p,\lambda) := \dot{q } \cdot p  - F(q,\dot q,\lambda)$ gives
\begin{align*}
H &= \dot{q } \cdot p  - F(q,\dot q,\lambda) \\
&= \dot{q } \cdot p  - \frac{1}{2} \dot{q }^T M \dot{q } +V(q )  + \sum_{i=1}^k \lambda_i g _i(q ) \cdot \dot{q } \\
&= \dot{q } \cdot \left( p  - \frac{1}{2} M \dot{q } + \sum_{i=1}^k \lambda_i g _i(q ) \right) + V(q ) \\
&= \dot{q } \cdot \left( M \dot{q } - \sum_{i=1}^k \lambda_i g _i(q )  - \frac{1}{2} M \dot{q } + \sum_{i=1}^k \lambda_i g _i(q ) \right) + V(q )  && \\
&= \dot{q } \cdot \left( \frac{1}{2} M \dot{q } \right) + V(q ) \\
&= \frac{1}{2} \left( p  + \sum_{i=1}^k \lambda_i g _i(q ) \right)^T M^{-1}  \left( p  + \sum_{i=1}^k \lambda_i g _i(q ) \right)  + V(q ). && 
\end{align*}
A calculation shows the equivalence of the right hand side of (\ref{eqn:hamqdot}) and $\nabla_p H$; of the right hand side of (\ref{eqn:hamsecond}) and $-\nabla_q H(q,p,\lambda)$; and of constraints $g_i(q)\cdot \dot q=0$ and $0 = \nabla_\lambda H(q,p,\lambda)$.

The constraints $0=\nabla_\lambda H(q,p,\lambda)$ are the following set
of equations linear in $\lambda$,
\begin{equation}
\begin{pmatrix}
   g _1 \cdot M^{-1} g _1 & \cdots & g _1 \cdot M^{-1} g _k \\
   \vdots & \ddots & \vdots \\
   g _k \cdot M^{-1} g _1 & \cdots & g _k \cdot M^{-1} g _k
\end{pmatrix}
\begin{pmatrix}
   \lambda_1 \\
   \vdots \\
   \lambda_k
\end{pmatrix} = -
\begin{pmatrix}
   g _1 \cdot M^{-1} p  \\
   \vdots \\
   g _k \cdot M^{-1} p 
\end{pmatrix} \label{eqn:lambdasystem}
\end{equation}
which has a unique solution for $\lambda$ for all $q$, $p$ because the matrix is $G M^{-1} G^T$ where $G$ is the $k\times n$ matrix whose $i$th row is $g_i^T$. The  assumption that the $g_i$ are linearly independent means that $G$ has full rank $k$ and hence that $G M^{-1} G^T$ is nonsingular. The constraints therefore have a unique solution for $\lambda$ that we write as
$\lambda=\tilde\lambda(q,p)$, that is, the constraint submanifold is a graph over $(q,p)$. Differentiating these constraints with respect to $t$ then yields ODEs for $\dot\lambda$, that is, the system (\ref{eq:ham}) has (differentiation) index 1. The symplectic form on $C$ is $\frac{1}{2}dz \wedge J dz = dq\wedge dp$.
\qquad \end{proof}

We emphasize that although Proposition \ref{thm:lagham} is not original, the usual treatment is to go one
step further and eliminate the Lagrange multipliers $\lambda$ to get a canonical Hamiltonian system in $(q,p)$
(\cite{bloch}, Thm. 7.3.1). This step may not be desirable either analytically or numerically.

Under certain conditions, namely that 
the Legendre transform
that defines the conjugate momenta is invertible to give
$\dot{q }$, 
Proposition~\ref{thm:lagham} can be generalized to allow a
general Lagrangian and general constraints. A very thorough
geometric treatment of this type of constraint, applying the 
Gotay--Nestor geometric version of the Dirac--Bergmann constraint algorithm,
can be found in \cite{ma-co-de}.  The proof of the following proposition follows the same lines as Proposition \ref{thm:lagham} and is omitted.


\begin{prop}\label{thm:laghamgeneral}
   If the Legendre transform mapping $(\dot{q }, q , \lambda  ) \rightarrow
   (p , q , \lambda  )$ given in equation~(\ref{eqn:hampdfn3}) is invertible then
   the Euler--Lagrange equations for the action
   \[
   S(q )=\int_{t_0}^{t_1} L(t, q , \dot{q }) dt
   \]
   subject to the constraints $g_i(q ,\dot{q }) = 0, i = 1, \ldots, k$ are equivalent
   to the generalized Hamiltonian system
      \begin{equation} \label{eq:genh2}
      J\dot{z } = \nabla H(z )
   \end{equation}
   where
   \begin{align}
      J &= \begin{pmatrix}0 & - \mathbb{I}_{n \times n} & 0\\
         \mathbb{I}_{n \times n} & 0 & 0\\0 & 0 & 0_{k \times k}\end{pmatrix},\quad
   z  = \begin{pmatrix} q  \\ p  \\ \lambda   \end{pmatrix} , \quad
      p  = \nabla_{\dot{q }} F(q , \dot{q }, \lambda  ), \label{eqn:hampdfn3} \\
      H(z ) &= \dot{q } \cdot p  - F( q , \dot{q },\lambda  ),\quad
      F( q ,\dot{q }, \lambda  ) = L(t, q , \dot{q }) - \sum_{i=1}^k \lambda_i g_i(q ,\dot{q }). \nonumber  
   \end{align}
If, in addition, the matrix $G(q,\dot q)$ given by $G_{ij} = \partial g_i(q,\dot q)/\partial \dot q_j$ has full rank $k$ for all $q$, $\dot q$, then the system of Eq. (\ref{eq:genh2}) has index 1, i.e., can be solved for $\lambda=\tilde \lambda(q,p)$.

\end{prop}

Proposition \ref{thm:laghamgeneral} can be generalized further, to any singular Lagrangian $L(q,\dot q,\lambda)$, and still
further to Lagrangians $L(z,\dot z)$ where $|L_{\dot z\dot z}|=0$, but the required nondegeneracy assumptions are not as geometrically transparent as those in Proposition~\ref{thm:laghamgeneral}.

\section{Symplectic integrators for generalized Hamiltonian systems}
\label{sec:integrators}
The Hamiltonian form (\ref{eq:ham}) suggests considering generalized Hamiltonian systems of the form 
\begin{equation}
\label{eqn:degham}
J\dot z = \nabla H(z),\quad z\in \mathbb{R}^m,
\end{equation}
where $J$ is a constant antisymmetric matrix, and we do not specify the constraints. Note that many kinds of constrained Hamiltonian systems (including those with holonomic constraints) can be written in this form; the constraint manifold $C$ is constructed as the subset of initial conditions for which the equations have a solution. In general,  these equations may not have solutions for all initial conditions; in the extreme case $J=0$, the equations are purely algebraic. However, it is easily seen that any solutions that do exist do preserve the (`pre-symplectic') 2-form $\frac{1}{2}dz\wedge J dz$, which is degenerate when $J$ is singular---this does not require the invertibility of $J$.
\begin{lemma} Any solutions to Eq. (\ref{eqn:degham}) preserve the 2-form $\frac{1}{2}dz\wedge J dz$.
\end{lemma}

{\em Proof.}
We have
\begin{align*}
{\textstyle\frac{1}{2}} (dz\wedge J dz)_t 
& = {\textstyle\frac{1}{2}}(dz\wedge J dz_t + dz_t \wedge J dz) \\
&= dz\wedge J dz_t \\
&= dz\wedge H_{zz}(z) dz \\
&= 0. \qquad\endproof
\end{align*}

In the particular case of Proposition \ref{thm:lagham}, the generalized Hamiltonian system that is obtained is equivalent to a canonical Hamiltonian system obtained by eliminating the Lagrange multipliers $\lambda$. Let $\lambda=\tilde\lambda(q,p)$ be the solution
to (\ref{eqn:lambdasystem}). Then Hamilton's equations for 
$\tilde{H}(q , p ) :=  H (q , p ,
{\tilde \lambda  }(q , p ))$ are 
\begin{align*}
\dot q_i & = \frac{\partial\tilde{H}}{\partial p_i}(q,p)\\
& = \frac{\partial H}{\partial p_i}(q,p,\tilde \lambda(q,p)) + \sum_{j=1}^k \frac{\partial H}{\partial \lambda_j}(q,p,\tilde \lambda(q,p))
\frac{\partial \tilde \lambda_j}{\partial p_i}(q,p)\\
& = \frac{\partial H}{\partial p_i}(q,p,\tilde \lambda(q,p)) 
\end{align*}
(and similarly for $\dot p$) which, together with $\frac{\partial H}{\partial \lambda}(q,p,\tilde\lambda(q,p))=0$, are equivalent to (\ref{eq:ham}).
That is, the two operations of eliminating the Lagrange multipliers and mapping the Hamiltonian to its Hamiltonian vector field commute; this can also be seen abstractly by considering the symplectic manifold $C$ with canonical coordinates $(q,p)$, symplectic form $dq\wedge dp$, and Hamiltonian $i^*H$.

Certain Runge--Kutta methods, e.g. the midpoint rule, are known to be symplectic when the
structure matrix $J$ is invertible~\cite{simhamdyn}.  However, as for the continuous time case, $J$ need not be invertible.

\begin{prop}\label{thm:symrk}
Any solutions of any symplectic Runge--Kutta method applied to
$J\dot{z } = \nabla H$ preserve the 2-form
$\frac{1}{2}dz\wedge J dz$, where $J$ is any constant antisymmetric matrix.
\end{prop}

\begin{proof} The $s$ stage symplectic Runge-Kutta method is
\begin{eqnarray}
J Z _{i} & = & J z _{0}+\Delta t\sum_{j=1}^{s}a_{ij} J F _j,\label{eqn:symplecticrkfirst}\\
J z _{1} & = & J z _{0}+\Delta t\sum_{j=1}^{s}b_{j} J F _j, \label{eqn:symplecticrksecond}
\end{eqnarray}
where
\begin{equation}
   J F _j = \nabla H(Z _j). \label{eqn:symplecticrkF}
\end{equation}
Here $\Delta t$ is the time step and the method maps $z_0$ to $z_1$.
The coefficients of a symplectic Runge--Kutta method obey
\begin{equation}
b_{i}b_{j}-b_{j}a_{ji}-b_{i}a_{ij} = 0. \label{eqn:symplecticrkcond}
\end{equation}
Taking the exterior derivative of equations~(\ref{eqn:symplecticrkfirst}),~(\ref{eqn:symplecticrksecond}), and~(\ref{eqn:symplecticrkF}) gives
\begin{gather}
   J dz _0 = JdZ _{i}-\Delta t\sum_{j=1}^{s}a_{ij}JdF _{j}, \label{eqn:symplecticrkdfirst} \\
   J dz _1 = Jdz _0-\Delta t\sum_{j=1}^{s}b_jJdF _j, \label{eqn:symplecticrkdsecond} \\
   J dF _j = H_{zz}(Z_j) dZ _j.   \label{eqn:symplecticrkdF}
\end{gather}
From equation~(\ref{eqn:symplecticrkdF}),
\begin{equation}
dZ _j \wedge JdF _j=dZ _j \wedge H_{zz}(Z_j)dZ _j=0. \label{eqn:symplecticrkzero}
\end{equation}
Substitution now gives, in the same way as in the original study of B-stability by Burrage and
Butcher \cite{bu-bu}, 
\begin{align*}
dz_1 \wedge J dz_1 &= dz_0 \wedge J dz_0 + 2\Delta t \sum_{j=1}^s dZ_j\wedge JdF_j + \\
& \qquad (\Delta t)^2
\sum_{i,j=1}^s(b_i b_j - b_j a_{ji} - b_i a_{ij})dF_i \wedge J dF_j \\
&= dz_0\wedge J dz_0.
\end{align*}
using \eqref{eqn:symplecticrkcond} and \eqref{eqn:symplecticrkzero}.
This establishes the proposition.
\end{proof}

Note that the underlying structure can be seen very clearly in the case of the midpoint rule
\begin{equation}
   \frac{Jz _1 - Jz _0}{\Delta t} = \nabla H \left(\frac{z _0 + z _1}{2}\right) := \nabla H(\bar z) .\label{eqn:ime}
\end{equation}
for which
\begin{align*}
dz _1  \wedge J dz _1 - dz _0 \wedge J dz _0
& = (dz _0 + dz _1) \wedge J (dz _1 - dz _0) \cr
& = (dz _0 + dz _1) \wedge \frac{1}{2} \Delta t H_{zz}(\bar z) (dz _0 + dz _1)\cr
& = 0.
\end{align*}

A full study of the geometry of the relations $(z_0,z_1)$ generated in Proposition \ref{thm:symrk} remains to be undertaken.\footnote{The relations generated in Proposition \ref{thm:symrk} are a generalization of the Viterbo generating functions used in symplectic topology \cite{viterbo}. These take the form $S\colon Q\times\mathbb{R}^k\to\mathbb{R}$; the submanifold
$p= S_q(q,\lambda)$, $0 = S_\lambda(q,\lambda)$ is Lagrangian in $T^*Q$. The parameters $\lambda$ allow the representation
of larger classes of Lagrangian submanifolds than the standard generating function $S(q)$ which generates $p=S_q(q)$ which is necessarily a graph over $Q$.} Unfortunately, the relations $(z_0,z_1)$ in Proposition \ref{thm:symrk} do not yield good integrators for arbitrary $J$ and $H$. For example, holonomic constraints can be specified as generalized Hamiltonian systems with $H=\tilde H(q,p) + \sum_{i=1}^k \lambda_i h_i(q)$. In this case the midpoint rule, say, generates maps from {\em all} $(q_0,p_0)$ to 
$(q_1,p_1)$ with the constraints satisfied at the midpoint. Not only is the phase space `wrong', this method is known to be not convergent in general \cite{hairer-jay}. The situation is much better for index 1 constraints.

\section{Symplectic integrators for index 1 constraints}
\ \\[-4.5mm]
\begin{prop}
\label{thm:4}
Let $J$ be any constant antisymmetric matrix and let $H$ be a Hamiltonian such that the generalized Hamiltonian system
\begin{equation}
\label{eqn:indexone}
J \dot z = \nabla H(z)
\end{equation}
has index 1, i.e., such that when the system is written in Darboux coordinates $(q,p,\lambda)$, the constraint $H_\lambda = 0$ has a unique solution for $\lambda$ for all $q$ and $p$. Then any symplectic Runge--Kutta method (\ref{eqn:symplecticrkfirst})--(\ref{eqn:symplecticrkcond}) applied to (\ref{eqn:indexone}) is well-defined for sufficiently small $\Delta t$, convergent of the same order as the Runge--Kutta method, preserves the constraint submanifold, and preserves the symplectic form on the constraint submanifold.
\end{prop}
\begin{proof}
By linear covariance of Runge--Kutta methods we can assume that $J$ is in Darboux form (although the theorem holds in any basis). 
Then the constraint part of the Runge--Kutta equations read
\begin{align*}
0 & = \nabla_\lambda H(Q_i,P_i,\Lambda_i),\ i=1,\dots,k, \\
0\frac{\lambda_1-\lambda_0}{\Delta t} &= \sum_{i=1}^s b_i\nabla_\lambda H(Q_i,P_i,\Lambda_i).
\end{align*}
Therefore the Lagrange multipliers $\Lambda_i$ at each stage
are given by the {\em exact} Lagrange multipliers evaluated at
$(Q_i,P_i)$, i.e. $\Lambda_i = \tilde \lambda(Q_i,P_i)$, and $\lambda_1$ is arbitrary. For convenience, we add the extra equations $\lambda_0 = \tilde \lambda(q_0,p_0)$, $\lambda_1=\tilde \lambda(q_1,p_1)$ which do not affect the method at all. The resulting method is equivalent to that obtained by eliminating the Lagrange multipliers in the Hamiltonian, applying a symplectic Runge--Kutta method, and lifting back to the constraint manifold by $\lambda=\tilde \lambda(q,p)$. It is therefore well defined for sufficiently small $\Delta t$ and convergent of the same order as the Runge--Kutta method. Because $\frac{1}{2}dz\wedge J dz = dq \wedge dp$, the symplectic form $dq\wedge dp$ is preserved on the constraint manifold.
\qquad \end{proof}

Note that the assumptions are satisfied if $|H_{\lambda\lambda}|\ne 0$. The constraints may be nonlinear in $\lambda$, and need not be solved analytically; the entire Runge--Kutta system for $(Q_i,P_i,\Lambda_i)$ can be numerically solved simultaneously.

\begin{cor}
Symplectic Runge--Kutta methods yield convergent constraint-preserving symplectic integrators for the Hamiltonian formulation of the Lagrange and sub-Riemannian problems given in Propositions \ref{thm:lagham} and \ref{thm:laghamgeneral}. When velocities are calculated using the Legendre transform, the constraints
$g_i(q)\cdot \dot q= 0$ (resp. $g(q,\dot q)=0)$ are satisfied exactly at the stages, and if the endpoint Lagrange multipliers are defined by $H_\lambda(q_n,p_n,\lambda_n)=0$, then the constraints are satisfied exactly at the endpoints.
\end{cor}

\section{Variational problems with holonomic and nonholonomic constraints}
\label{sec:holo}

Proposition~\ref{thm:lagham} allowed a nonholonomic variational problem to be converted into an index 1 constrained Hamiltonian
system that can be integrated using the symplectic midpoint rule.
In this section we show that if \emph{holonomic} constraints are
added to the original variational problem, then the resulting
Hamiltonian system is a simple \emph{holonomically} constrained system.   This system can
be solved by a symplectic method such as
\textsc{rattle}~\cite{mc-mo-ve-wi,Skeel94symplecticnumerical}.

\begin{prop}\label{thm:laghamholo}
   Let $M$ be a symmetric nonsingular $n \times n$ mass matrix,
   $V: \mathbb{R}^n \rightarrow \mathbb{R}$ a smooth
   potential, $g_i:\mathbb{R}^n\to\mathbb{R}^n$, $i=1,\dots,k$ be $k$ smooth functions , and $q $ be  a smooth extremal with fixed endpoints for the functional
   \begin{equation}\label{eq:action2}
   S({q})=\int_{t_0}^{t_1} L(t, {q}, \dot{{q}}) dt = \int_{t_0}^{t_1} \left( \frac{1}{2} \dot{{q}}^T M \dot{{q}} - V({q}) \right) dt
   \end{equation}
   subject to the {\em velocity} constraints ${g}_i({q}) \cdot \dot{{q}} = 0, i = 1, \ldots, k$ and the \emph{holonomic}
   constraints $h_i({q})  = 0, i = 1, \ldots, l$.
   Then
   \begin{align}\label{eq:hamhol}
      J\dot{{z}} &= \nabla H({z}) 
   \end{align}
   where
   \begin{align*}
      J &= \begin{pmatrix}0 & - \mathbb{I}_{n \times n} & 0 & 0\\
         \mathbb{I}_{n \times n} & 0 & 0 & 0\\0 & 0 & 0_{k \times k} & 0 \\ 0 & 0 & 0 & 0_{l\times l}\end{pmatrix},\quad
   {z} = \begin{pmatrix} {q} \\ {p} \\{\lambda}\\{\lambda^h} \end{pmatrix}, \\
      {p}  &=  M \dot{{q}} - \sum_{i=1}^k \lambda_i {g}_i({q}) \\
      H({z}) &= \frac{1}{2} \left( {p} + \sum_{i=1}^k \lambda_i {g}_i({q}) \right)^T M^{-1}  \left( {p} + \sum_{i=1}^k \lambda_i {g}_i({q}) \right)  + V({q})  + \sum_{i=1}^l \lambda_i^h h_i({q}) 
   \end{align*}
and, furthermore, the Euler--Lagrange equations for (\ref{eq:action2}) are equivalent to the generalized Hamiltonian system (\ref{eq:hamhol}).
If, in addition, the velocity constraints are linearly independent for all $q$, then Eq. (\ref{eq:hamhol}) is equivalent to a canonical holonomically constrained Hamiltonian system.

\end{prop}

\begin{proof}
As in Proposition \ref{thm:lagham} the extended Lagrangian $F$, the conjugate momenta $p$, and the Hamiltonian $H(q,p,\lambda,\lambda^h)$ are defined by
\begin{align*}
F &:= \frac{1}{2} \dot{{q}}^T M \dot{{q}} - V({q}) -  \sum_{i=1}^k \lambda_i {g}_i({q}) \cdot \dot{{q}} - \sum_{i=1}^l \lambda_i^h h_i({q}),\\
   {p}  & :=  \nabla_{\dot{{q}}} F  = M \dot{{q}} - \sum_{i=1}^k \lambda_i {g}_i({q}),\\
H &:= \dot{{q}} \cdot {p} - F. 
\end{align*}
The rest of the proof is a calculation along the same lines as for Proposition \ref{thm:lagham}.
\qquad \end{proof}

\begin{prop}\label{thm:rattlesym}
  Subject to standard nondegeneracy assumptions on the Hamiltonian, 
  the following algorithm yields a convergent, second order integrator that is symplectic on the constraint manifold defined by the (primary) holonomic constraints and the secondary constraints induced by them: (i) apply {\sc rattle} using the holonomic constraints; (ii) in the inner step of {\sc rattle}, when a time step of the unconstrained system is required, apply the midpoint rule to the generalized Hamiltonian system with Hamiltonian $H(q,p,\lambda,0)$.
\end{prop}
\begin{proof}
 Eliminating the velocity constraints by solving for the Lagrange multipliers $\lambda_i$ yields a standard holonomically constrained system. Applying {\sc rattle} (with the midpoint rule in the inner step) to this system yields a convergent second order integrator on the constraint surface. Applying the midpoint rule in the inner step is equivalent to applying the midpoint rule to the generalized Hamiltonian system with Hamiltonian $H(q,p,\lambda,0)$.
 \qquad \end{proof}

\section{Example: Sub-Riemannian geodesics} \label{sec:srexample}

Trajectories of a two-wheeled vehicle with a  front steering wheel and a non-steering back wheel, moving on a smooth surface, 
will be modelled.  
We consider the two-wheeled vehicle shown in
Fig.~\ref{fig:vehicle_two_wheeled} with length $L$, 
back wheel at $(z, w)$, and front wheel at $(x, y)$.  The front wheel is
at an angle $\phi$ and the vehicle is at an angle $\theta$.
        \begin{figure}            \centering
            \includegraphics[width=3.3cm]{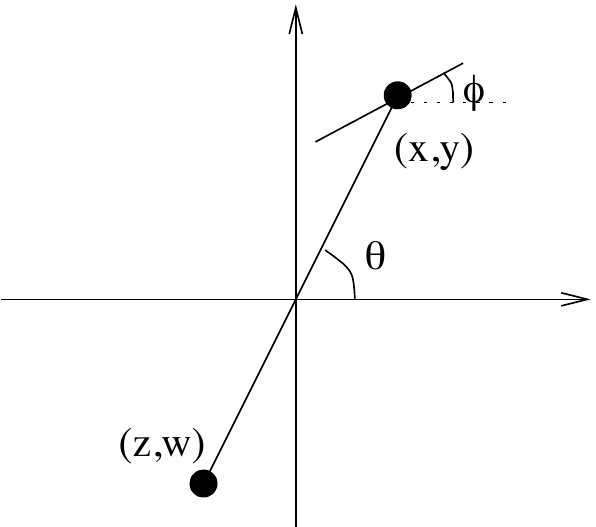}
            \caption{A two wheeled vehicle showing the front wheel angle $\phi$ and the vehicle angle of $\theta.$}
            \label{fig:vehicle_two_wheeled}
         \end{figure}

If the speed of the front wheel is $v$, its velocity of the front wheel must obey
\begin{equation*}
   \dot{x}  = v \cos \phi,\qquad
   \dot{y} =  v\sin \phi.
\end{equation*}
Eliminating $v$, the velocity of the front wheel obeys the constraint
\begin{equation}
  \dot{x} \sin \phi - \dot{y} \cos \phi = 0.    \label{eqn:vehicle_steering_constraint} 
  \end{equation}
  Similarly, the velocity of the back wheel obeys the constraint
\begin{equation}
   \dot{z} \sin \theta - \dot{w} \cos \theta = 0.    \label{eqn:vehicle_driving_constraint}
\end{equation}
We can eliminate equation~(\ref{eqn:vehicle_driving_constraint}), and
thus the variables $z$ and $w$, using the distance between the two wheels which
relates the four variables.  Notice that
\begin{equation*}
   x - z = L \cos \theta,\qquad
   y - w = L \sin \theta 
\end{equation*}
so that
\begin{equation*}
   \dot{z} = \dot{x} + L \dot\theta \sin \theta,\qquad 
   \dot{w}= \dot{y} - L \dot\theta \cos \theta 
\end{equation*}
which substituted into Eq.~(\ref{eqn:vehicle_driving_constraint})
gives
\begin{equation}
\dot{x} \sin \theta  - \dot{y} \cos \theta + L \dot{\theta} = 0 \label{eqn:vehicle_driving_constraint_simplified}
\end{equation}

We take the Lagrangian to be
\begin{equation}
\label{eqn:L2}
   L = \frac{1}{2}\left( \dot{x}^2 + \dot{y}^2 + \alpha \dot{\theta}^2 + \beta \dot{\phi}^2 \right) - V(x,y)
\end{equation}
where the potential $V(x,y)$ is the (scaled) height of the surface, which
together with the constraints (\ref{eqn:vehicle_steering_constraint}) and~(\ref{eqn:vehicle_driving_constraint_simplified})
gives an index 1 system as in Proposition \ref{thm:lagham}. That is, when $V=0$ we
are calculating geodesics of the sub-Riemannian metric defined by
Eqs. (\ref{eqn:L2}), (\ref{eqn:vehicle_steering_constraint}) and~(\ref{eqn:vehicle_driving_constraint_simplified}). 
To put it another way, we are seeking shortest paths that move the vehicle from one configuration to another subject to the constraints of its geometry---the `parallel parking' problem.
In the numerics, we use the midpoint rule.


We first used the potential $V(q) = -\cos r$, where $r$ is the midpoint of the vehicle,
and numerically checked the second-order convergence of the method,
 to a reference solution computed by {\sc matlab}'s {\tt ode15s},
numerical conservation of the symplectic form, exact conservation
of the original constraints (up to round-off error), and behaviour
of the energy error. A sample result is shown in Fig.~\ref{fig:energy}, 
from which the energy errors appear to be bounded,
as expected for a symplectic integrator.

        \begin{figure}[t]            \centering
            \includegraphics[width=4.5cm]{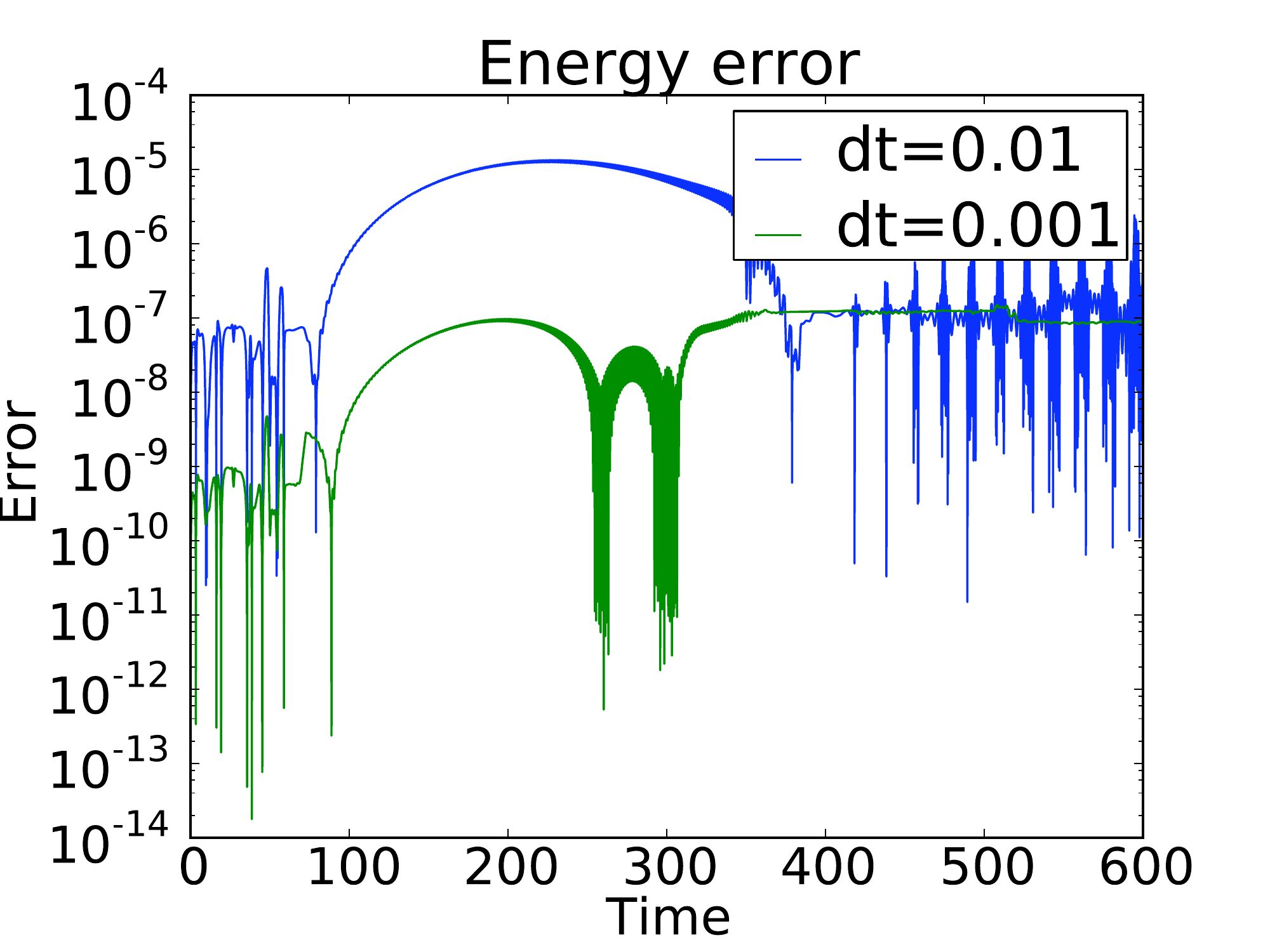}
            \input{energy_caption}
            \label{fig:energy}
         \end{figure}
         
We then studied special solutions of the free motion case $V=0$, which has two simple solutions that are relative
equilibria for the translation and rotation symmetries of the problem,
namely straight line and circular motion.
If the vehicle
starts with $\theta = \phi = 0$, the exact solution is a straight line motion. Let $\theta = \phi = 0$, $\dot{x} = 1$, and $\dot{y} = 0$.  The
constraints in equations~(\ref{eqn:vehicle_steering_constraint})
and~(\ref{eqn:vehicle_driving_constraint_simplified}) are satisfied.
Equation~(\ref{eqn:hamqdot}) gives the initial generalized momenta
values: all are zero except $p_x = 1$. The discretization gives the exact solution; 
however, as the solution is unstable, round-off errors eventually cause the
vehicle to wander.

For the circular motion, let $\theta = at$, $\phi = at + \frac{\pi}{2}$, $\dot{x} = -c
\sin(\theta)$, and $\dot{y} = c \cos(\theta)$.  There are two
constants, $a$ and $c$, to be determined.  
Equation~(\ref{eqn:vehicle_steering_constraint}) gives
$\lambda   = (1, -c)$, and
equation~(\ref{eqn:vehicle_driving_constraint_simplified}) gives
$aL = c$.  Using these values in equation~(\ref{eqn:hamp}) gives
the initial generalized momenta values: $(p_x, p_y, p_{\theta}, p_{\phi}) = (0, 0,
a(1+L^2), a)$.  For this simple trajectory $a$ is chosen to be $1$.
In Fig.~\ref{fig:circle_trajectory} the circle trajectory of
the vehicle is confirmed.
        \begin{figure}[t]            \centering
            \includegraphics[width=4.4cm,height=4.18cm]{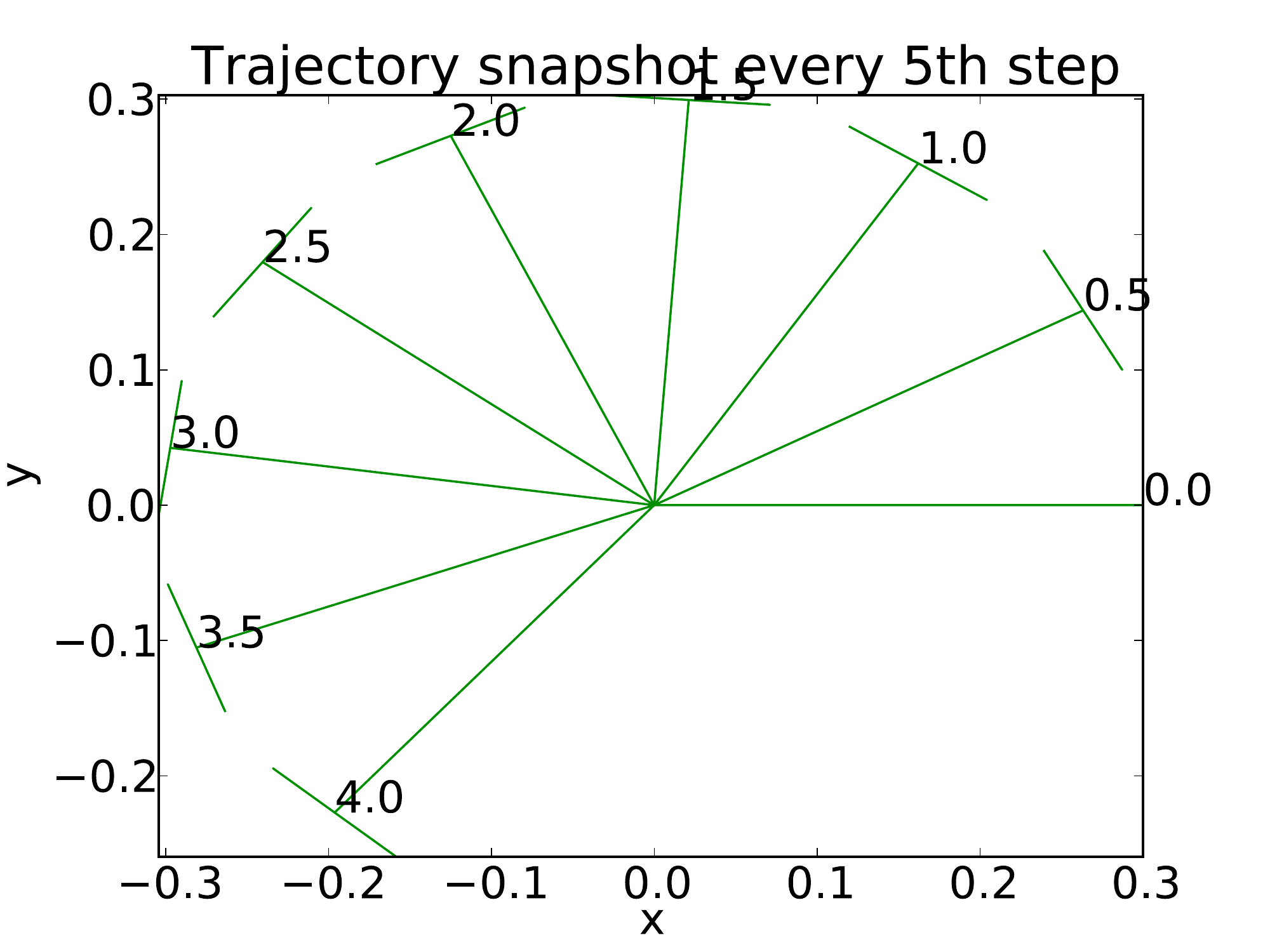}
              \includegraphics[width=5.6cm]{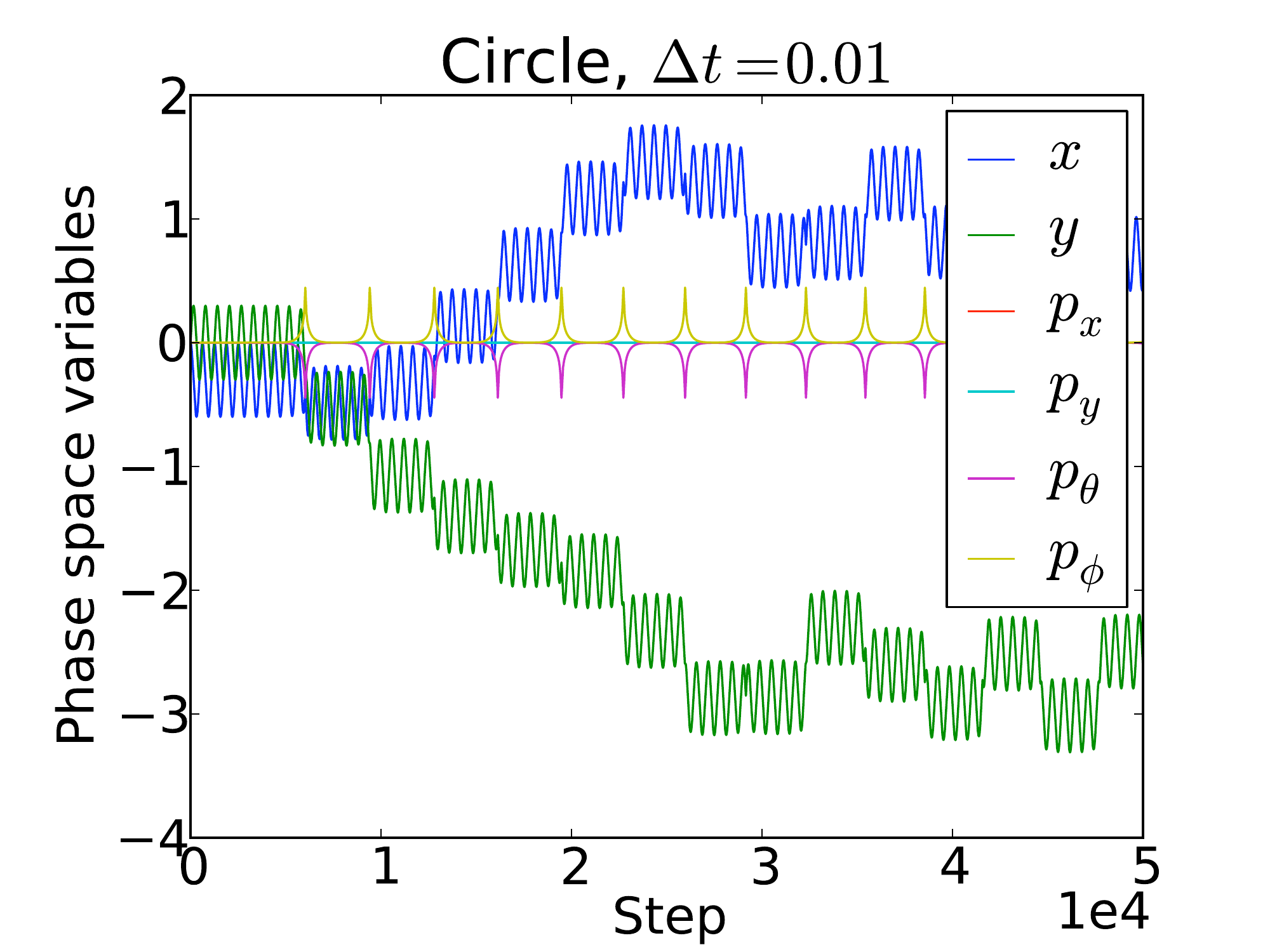}
          \input{circle_trajectory_caption}
           \label{fig:circle_trajectory}
         \end{figure}
If the trajectory is computed for larger times
the vehicle leaves the circle; the solution appears to be
unstable, but, interestingly, appears to repeatedly return 
to the circular orbit, indicating a possible relative homoclinic structure
in this problem. 


         \section{Example: the Heisenberg problem} \label{sec:hexample}

A previous study of geometric integrators for sub-Riemannian variational problems used
a discrete variational approach to obtain constrained symplectic integrators \cite{benito:083521}. Our approach,
applying symplectic integrators to the Hamiltonian formulation, yields geometric integrators
with the same geometric properties, but uses standard integrators that allow any order with standard 
implementations, and does not require an approximation of $\dot q$, that is, it naturally yields first-order
trajectories in $(q,p)$ instead of second-order trajectories in $q$.

We repeat the numerical illustration of \cite[p.~12]{benito:083521}, the Heisenberg problem, using our approach.
This is to 
find the extremal ${q}(t) = (x(t), y(t), z(t))$ of
   \[
   S({q})=\int_{t_0}^{t_1} L(t, {q}, \dot{{q}}) dt = \int_{t_0}^{t_1} \left( \frac{1}{2} \dot{{q}}^T \dot{{q}} - V({q})\right) dt
   \]
   subject to the constraint ${g}({q}) \cdot
   \dot{{q}} = 0$, where ${g}({q}) = (-y, x, 1)$.

Equation~(\ref{eqn:hamqdot}) gives $\dot{{q}}$:
\begin{equation}
\begin{pmatrix} \dot{x} \\ \dot{y} \\ \dot{z} \end{pmatrix} =
      \begin{pmatrix}p_x \\ p_y \\ p_z \end{pmatrix} +
         \lambda
         \begin{pmatrix} -y \\ x \\ 1 \end{pmatrix}.
      \label{eqn:heisenberg_top_rhs}
\end{equation}
Using equation~(\ref{eqn:hamsecond}), $\dot{{p}}$ can be written
\begin{equation} \label{eqn:vehicle_pdot}
   \begin{pmatrix} \dot{p_x} \\ \dot{p_y} \\ \dot{p_z} \end{pmatrix} =
       - \nabla V({q}) -
         \lambda \begin{pmatrix} 0 & 1 & 0 \\ -1 & 0 & 0 \\ 0 & 0 & 0  \end{pmatrix}
      \left[
      \begin{pmatrix}p_x \\ p_y \\ p_z \end{pmatrix} +
         \lambda
         \begin{pmatrix} -y \\ x \\ 1 \end{pmatrix}
      \right]
\end{equation} \label{eqn:heisenberg_pdot}
and we have the constraint
${g} \cdot \left( {p}  + \lambda {g} \right) = 0$, which gives
\[
\lambda = -\frac{{g} \cdot  {p}}{{g} \cdot {g}}.
\]

A simple trajectory starting with the same initial conditions as in
~\cite[pg.~15]{benito:083521} is shown in
Fig.~\ref{fig:heisenberg_trajectory}.  Their initial conditions
are $(x, y, z, \dot{x}, \dot{y}, \dot{z}, \lambda) = (0, 0, 0, 0.1, 0.3, 0, 1)$,
which when converted to generalized momenta variables are
$(x, y, z, p_x, p_y, p_z, \lambda) = (0, 0, 0, 0.1, 0.3, 1, 1)$.  
The results are consistent with~\cite[p. 14]{benito:083521}.

        \begin{figure}[t]            \centering
            \includegraphics[width=2in]{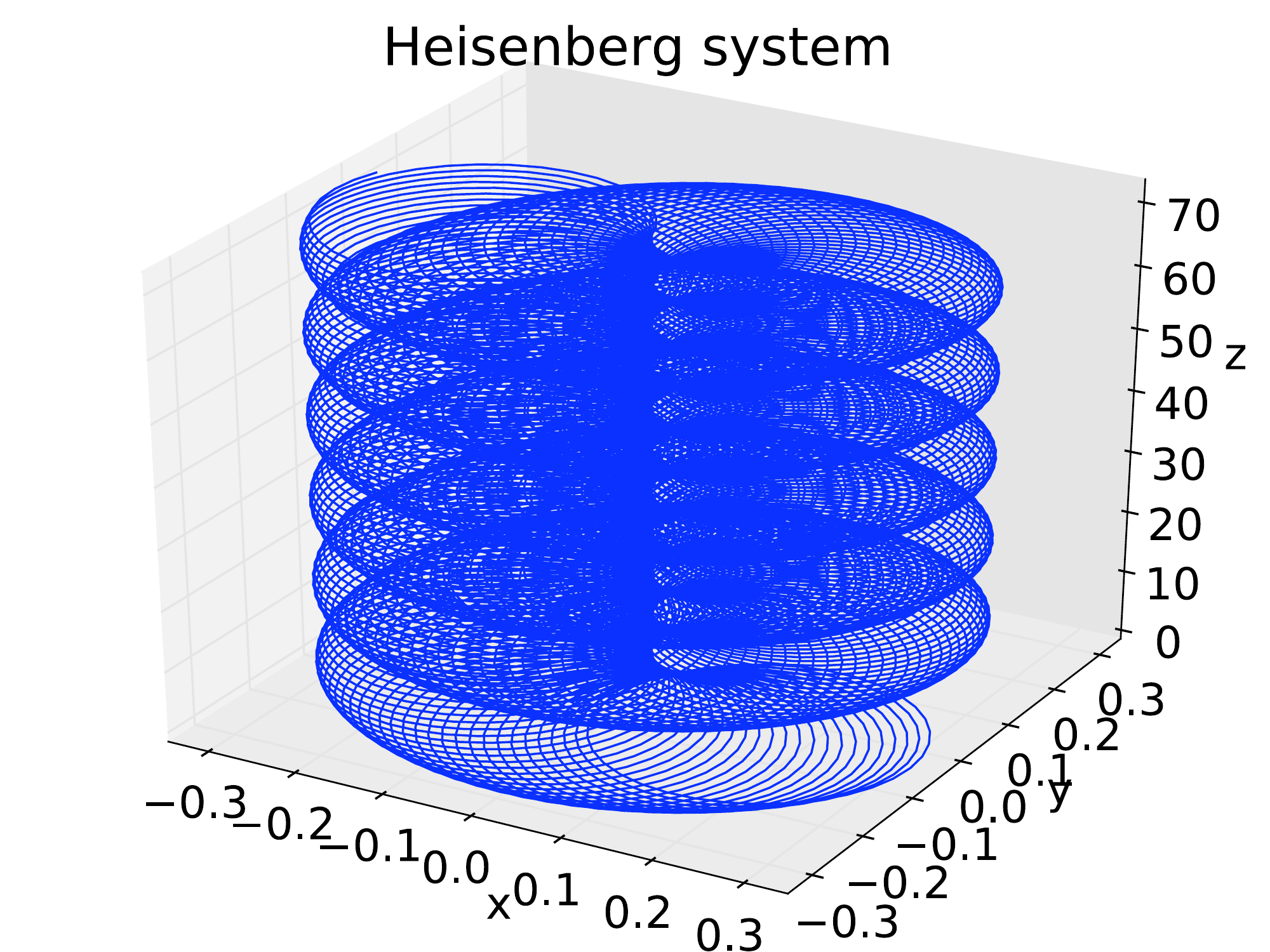}
            \input{heisenberg_trajectory_caption}
            \label{fig:heisenberg_trajectory}
         \end{figure}

\section{Discussion}
We have constructed symplectic integrators for a different class of con strained Hamiltonian systems than the holonomic constraints most commonly considered in the literature. The class includes important practical problems arising in sub-Riemannian geometry. We have restricted our attention to symplectic Runge--Kutta methods;  a generalization to partitioned methods in which different Runge--Kutta coefficients are used for $q$, for $p$, and for $\lambda$ is straightforward. In other work \cite{mc-mo-ve-wi}, we reinterpret these methods as an instance of {\sc rattle} in an extended phase space; that point of view also suggests different generalisations. 

We note that the nondegeneracy conditions in Propositions \ref{thm:lagham}, \ref{thm:laghamgeneral}, and \ref{thm:laghamholo} are essential for the integrators in Proposition \ref{thm:4},  indeed, for the entire approach, to work. It is not clear to what extent the approach can be extended to handle more general constraints, for example, to the system
$J \dot z = \nabla H + \lambda \nabla g$,  where the constraint submanifold $g(z)=0$ is symplectic. No symplectic, constraint-preserving method is known for this problem.
As remarked before Proposition \ref{thm:4}, 
a full study of the geometry of the relations $(z_0,z_1)$ generated in Proposition \ref{thm:symrk} remains to be undertaken. 
Any solutions are symplectic, so this gives access to a much larger class of symplectic maps than do traditional
generating functions. Note that new variables (analogous to $\lambda$) can be added as needed to generate
larger classes of maps. 

\section*{Acknowledgements}

O.~Verdier would like to acknowledge the support of the \href{http://wiki.math.ntnu.no/genuin}{GeNuIn Project}, funded by the \href{http://www.forskningsradet.no/}{Research Council of Norway},  as well as the hospitality of the \href{http://ifs.massey.ac.nz/}{Institute for Fundamental Sciences} of Massey University, New Zealand, where some of this research was conducted.
K.~Modin would like to thank the \href{http://www.math.ntnu.no}{Department of Mathematics at NTNU} in Trondheim and the \href{http://www.vr.se}{Swedish Research Council} for support. This research was supported by the 
 \href{http://wiki.math.ntnu.no/crisp}{CRiSP Project}, funded by the \href{http://cordis.europa.eu/fp7/home_en.html}{European Commission's Seventh Framework Programme}, and by the Marsden Fund of the Royal Society of New Zealand.
 
\bibliographystyle{siam}
\bibliography{IndexOneRevised}

\end{document}